\documentclass{proc-l}
\usepackage{amssymb}
\usepackage{hyperref}
\hypersetup{
  colorlinks=true,
  linkcolor=black,
  citecolor=black,
  urlcolor=black
}
\usepackage{cleveref}
\usepackage{enumitem}
\usepackage{mathtools}
\usepackage{tikz-cd}
\usepackage{mathrsfs}
\usepackage{float}

\newtheorem{theorem}{Theorem}[section]
\newtheorem{lemma}[theorem]{Lemma}
\newtheorem{corollary}[theorem]{Corollary}
\theoremstyle{definition}

\newtheorem{proposition}[theorem]{Proposition}

\theoremstyle{remark}
\newtheorem{remark}[theorem]{Remark}

\numberwithin{equation}{section}

\begin{document}

\title{$\pi_1$-injective proper maps between non-compact surfaces}

\author[Sumanta Das]{Sumanta Das}
\address{Department of Mathematics, Indian Institute of Science, Bangalore 560012, India}
\curraddr{Department of Mathematics, Indian Institute of Technology Bombay, Mumbai 400076, India}
\email{sumantadas@iisc.ac.in}
\thanks{Supported by the National Board for Higher Mathematics PhD Scholarship and the IIT Bombay Institute Postdoctoral Fellowship.}
\keywords{Topological rigidity, infinite-type surfaces, finite-sheeted coverings}
\subjclass[2020]{Primary 57K20, Secondary 55S37.}

\date{}

\dedicatory{}

\begin{abstract}
We classify all $\pi_1$-injective proper maps between non-compact surfaces up to  proper homotopy.
\end{abstract}

\maketitle
\section{Introduction}
A \emph{surface} (resp. \emph{bordered surface}) is a connected, orientable $2$-dimensional manifold with an empty (resp. a non-empty) boundary. In 1927, Nielsen proved that any $\pi_1$-injective map between two compact surfaces is homotopic to a covering map. In dimension three, achieving an analog classification of $\pi_1$-injective proper maps is also possible, subject to various constraints. For example, Waldhausen \cite[Theorem 6.1]{MR224099} proved that if $f\colon N\to M$ is $\pi_1$-injective map between two connected, closed, orientable, irreducible $3$-manifolds, where $\pi_1(N)\neq 0$ and $M$ is Haken, then $f$ homotopic to a covering map. Brown and Tucker \cite[Theorem 4.2]{MR334225} showed that if $f\colon M\to N$ is a $\pi_1$-injective proper map between two connected, non-compact, orientable, irreducible, end-irreducible, boundaryless $3$-manifolds such that $\pi_1(M)$ is not isomorphic to the fundamental group of any compact surface, then $f$ is properly homotopic to a finite-sheeted covering map.

The main theorems of this article are the following.

\begin{theorem}
    \textup{Let $\Sigma',\Sigma$ be two non-compact oriented surfaces such that $\Sigma'$ is neither the plane nor the punctured plane. Suppose $f\colon \Sigma'\to \Sigma$ is a $\pi_1$-injective proper map. Then $f$ is properly homotopic to a finite-sheeted covering map. In particular, $\deg(f)\neq 0$. }\label{all}
\end{theorem}
\begin{theorem}\label{puncturedplane}
    \textup{Suppose $f$ is a $\pi_1$-injective proper map from $\mathbb{C}^* \coloneqq \mathbb{C} \setminus \{0\}$ to a non-compact oriented surface $\Sigma$. Let $n\coloneqq \deg(f)$. Then, we have the following.
    \begin{enumerate}
        \item\label{puncturedplane1} If $n=0$, then there exists a $\pi_1$-injective, proper embedding $\iota\colon \mathbb  S^1\times [0,\infty)\hookrightarrow \Sigma$, along with a non-zero integer $d$, such that after a proper homotopy, $f$ can be described by the proper map $\mathbb  S^1\times \mathbb  R\ni (z,t)\longmapsto \iota\left(z^d,|t|\right)\in \Sigma$. In particular, given any compact subset $K$ of $\Sigma$, there exists a proper map $g$ properly homotopic to $f$ such that $\operatorname{im}(g)\subseteq \Sigma\setminus K$.
        \item\label{puncturedplane2} If $n\neq 0$, then $\Sigma=\mathbb  C^*$ and $f$ is properly homotopic to the covering $\mathbb  C^*\ni z\longmapsto z^n\in \mathbb  C^*$ if $n>0$, and to the covering $\mathbb  C^*\ni z\longmapsto \overline z^{-n}\in \mathbb  C^*$ if $n<0$.
    \end{enumerate}}
\end{theorem}
\begin{theorem}\label{plane}
    \textup{Suppose $f$ is a proper map from $\mathbb  C$ to a non-compact oriented surface $\Sigma$. Let $n\coloneqq \deg(f)$.  Then, we have the following.\begin{enumerate}
        \item\label{plane1} If $n=0$, then for every compact subset $K$ of $\Sigma$, there exists a proper map $g$ properly homotopic to $f$ such that $\operatorname{im}(g)\subseteq \Sigma\setminus K$.
        \item\label{plane2} If $n\neq 0$, then $\Sigma=\mathbb  C$, and $f$ is properly homotopic to the branched covering $\mathbb  C\ni z\longmapsto z^n\in \mathbb  C$ if $n>0$, and to the branched covering $\mathbb  C\ni z\longmapsto \overline z^{-n}\in \mathbb  C$ if $n<0$.
        \end{enumerate}} 
\end{theorem}

\Cref{all} is also stated in \cite[Proposition 2.1(b)]{MR1308972}, with a reference to an unpublished preprint \cite{neverpublished}, where the proof is described as being in the spirit of a result from \cite[Theorem 4.2]{MR334225}. However, we would like to note that our proof of \Cref{all} follows an entirely different approach and does not depend on the methods or ideas presented in these works.

\section{Preliminaries}
\subsection{Conventions}
For integers \( g \geq 0 \), \( b \geq 0 \), and \( p \geq 0 \), denote by \( S_{g,b} \) the connected, oriented \( 2 \)-manifold of genus \( g \) with \( b \) boundary components. Let \( S_{g,b,p} \) denote the \( 2 \)-manifold obtained by removing \( p \) points from \( \operatorname{int}(S_{g,b}) \), where \( \operatorname{int}(M) \) denotes the interior of a manifold \( M \). For convenience, we will occasionally refer to \( S_{0,1} \), \( S_{0,2} \), \( S_{0,3} \), \(S_{1,1}\), \( S_{1,2} \), and \( S_{0,1,1} \) as a disk, an annulus, a pair of pants, a one-holed torus, a two-holed torus, and a punctured disk, respectively.

Let \( \Sigma \) be a non-compact surface, and let \( \mathrm{S} \) be a subsurface of \( \Sigma \). We say that \( \mathrm{S} \) is \emph{essential} if the inclusion-induced map \( \pi_1(\mathrm{S}) \to \pi_1(\Sigma) \) is injective, and that it is \emph{properly embedded} if the inclusion map \( \mathrm{S} \hookrightarrow \Sigma \) is proper.

\subsection{The space of ends}Let \( M \) be a non-compact connected topological manifold, and let \( a, b \colon [0,\infty) \to M \) be two proper maps. We say that \( a \) and \( b \) are \emph{equivalent} if, for every compact subset \( C \subseteq M \), there exists \( t_C \geq 0 \) such that \( a(t) \) and \( b(t) \) lie in the same component of \( M \setminus C \) for all \( t \geq t_C \). The equivalence class of a proper map \( a \colon [0,\infty) \to M \) is called an \emph{end} of \( M \), and is denoted by \( [a] \).

The set of all ends of \( M \), denoted \( \operatorname{Ends}(M) \), carries a natural topology whose basis consists of open sets of the form  
\[
\underline{A} \coloneqq \{ e \in \operatorname{Ends}(M) \mid \exists a \in e \text{ such that } a([t,\infty)) \subseteq A \text{ for some } t \geq 0 \},
\]
where \( A \) is a component of the complement of some compact subset of \( M \), and the closure of \( A \) in \( M \) is non-compact. Equipped with this topology, \( \operatorname{Ends}(M) \) is homeomorphic to a closed subset of the Cantor set.

If \( N \) is another non-compact connected topological manifold and \( f \colon M \to N \) is a proper map, then \( f \) induces a continuous map $\operatorname{Ends}(f) \colon \operatorname{Ends}(M) \to \operatorname{Ends}(N)$,
given by \( [a] \longmapsto [fa] \). Thus, \( \operatorname{Ends} \) defines a functor in the sense that the induced map of the identity is the identity, and the induced map of a composition of two proper maps is the composition of their induced maps. Moreover, if two proper maps \( f_0, f_1 \colon M \to N \) are properly homotopic, then \( \operatorname{Ends}(f_0) = \operatorname{Ends}(f_1) \) \cite[Proposition 3.3.12]{MR3598162}.

Let \( M \) be a surface, and let \( e \) be an end of \( M \). We say that \( e \) is \emph{planar} if there exists a basic open neighborhood \( \underline{A} \) of \( e \) such that \( A \) embeds in the plane. The end \( e \) is said to be \emph{isolated} if it is an isolated point in the space of ends \( \operatorname{Ends}(M) \). Thus, \( e \) is an isolated planar end of \( M \) if and only if there exists a properly embedded punctured disk \( D_* \subseteq M \) such that \( \{e\} = \underline{\operatorname{int}(D_*)} \). Observe that if \( M \) is not homeomorphic to the plane, then every properly embedded punctured disk in \( M \) is necessarily essential.

\subsection{The degree of a proper map}Consider a pair $A\subseteq X$ of topological spaces, and let $n$ be a non-negative integer. Recall that the $n$-th singular cohomology with compact support $H^n_c(X, A;\mathbb  Z)$ is equal to the direct limit $\varinjlim H^n\big(X, A\cup(X\setminus K);\mathbb  Z\big)$, where $K$ is a compact subset of $X$ and the various maps to define this direct system  are induced by inclusions. Hence, for a compact subset $K$ of $X$, the definition of direct limit yields an \emph{obvious map} $H^n\big(X, A\cup(X\setminus K);\mathbb  Z\big)\to H^n_c(X, A;\mathbb  Z)$. 

Let $M_1$ and $M_2$ be two connected, oriented, topological $n$-manifolds. Denote the preferred generator of the $n$-th singular cohomology with compact support $H^n_c(M_j,\partial M_j;\mathbb  Z)\cong \mathbb  Z$  compatible with the orientation of $M_j$ by $[M_j]$ for each $j=1,2$. If $f\colon (M_1, \partial M_1)\to (M_2, \partial M_2)$ is a proper map, then the \emph{degree} of $f$ is the unique integer $\deg(f)$ that satisfies $H^n_c(f)\big([M_2]\big)=\deg(f)\cdot [M_1]$. Note that the degree is proper homotopy invariant and multiplicative. We quote a few important theorems. By a disk \( D \) in a smooth \( n \)-manifold \( M \), we mean the image of \( \{ x \in \mathbb{R}^n \colon |x| \leq 1 \} \) under a smooth embedding \( \{ x \in \mathbb{R}^n \colon |x| \leq 2 \} \hookrightarrow M \).

\begin{theorem}
\textup{\cite[Lemma 2.1(b)]{MR192475} Let $f\colon M\to N$ be a proper map between two connected, oriented, smooth $n$-manifolds such that $f(\partial M)\subseteq\partial N$. Suppose there exists a disk $D$ in $\operatorname{int}(N)$ with the property that $f^{-1}(D)$ is the pairwise disjoint union of disks $D_1,..., D_k$ in $\operatorname{int}(M)$ ($k$ must be a non-negative integer as $f$ is proper) such that $f$ maps each $D_i$ homeomorphically onto $D$. For each $i$, let $\varepsilon_i$ be $+1$ or $-1$ according as the homeomorphism $f\vert D_i\to D$ is orientation-preserving or orientation-reversing. Then $\deg(f)=\sum_{i=1}^k\varepsilon_i$. In particular, if $f^{-1}(D)=\varnothing$, then $\deg(f)=0$.}
\label{degreeonemapchecking}\end{theorem}

\begin{theorem}\textup{\cite[Theorems 3.1 and 4.1]{MR192475} Let  $f\colon M\to N$ be a proper map between two connected, oriented, smooth $n$-manifolds such that $f(\partial M)\subseteq \partial N$. Let $\ell\coloneqq |\deg(f)|$. Then there is a proper map $g\colon M\to N$ with $g(\partial M)\subseteq \partial N$ and a homotopy $\mathcal H\colon M\times [0,1]\to N$ from $f$ to $g$ with the following properties:} 
\begin{itemize}
    \item \textup{There exists a compact subset $K\subseteq \operatorname{int}(M)$ such that $\mathcal H(x,t)=f(x)$ for all $(x,t)\in (M\setminus K)\times [0,1]$. In particular,  $\mathcal H$ is a proper homotopy relative to $\partial M$. }
    \item\textup{There exists a disk $D\subseteq \operatorname{int}(N)$ such that $g^{-1}(D)$ is the pairwise disjoint union of disks $D_1,..., D_\ell$ in $\operatorname{int}(M)$ and $g$ maps each $D_i$ homeomorphically onto $D$.}
\end{itemize}
\textup{Therefore, if the degree of $g$, which is the same as the degree of $f$, is positive (resp. negative), then  $g\vert D_i\to D$ is an orientation-preserving (resp. orientation-reserving) homeomorphism for each $i=1,...,\ell$. In particular, if $\ell=0$, then $f$ can be properly homotoped to a non-surjective map,  relative to the complement of a compact subset of $\operatorname{int}(M)$. }
\label{reverseprocessindegreefinding}
\end{theorem}

\begin{theorem}
     \textup{\cite[Corollary 3.4]{MR192475} Let $f\colon M\to N$ be a proper map between two connected, oriented, smooth $n$-manifolds such that $f(\partial M)\subseteq \partial N$. If $\deg(f)\neq 0$, then the index $\left[\pi_1(N):\operatorname{im}\pi_1(f)\right]$ divides $\deg(f)$. In particular, if $\deg(f)=\pm 1$, then $\pi_1(f)$ is surjective.}\label{degreeonemapsarepi1surjective}
 \end{theorem}
The proof of \Cref{degreeonemapsarepi1surjective} uses the following lemma. 
\begin{lemma}\label{techlemma}
    \textup{\cite[Proof of Theorem 3.1]{MR192475} Let $f\colon M\to N$ be a proper map between connected, oriented, smooth $n$-manifolds such that $f(\partial M)\subseteq \partial N$, and let $p\colon \widetilde N\to N$ be the covering corresponding to the subgroup $\operatorname{im}\pi_1(f)$ of $\pi_1(N)$. If $\deg(f)\neq 0$, then $p$ is a proper map. }
\end{lemma}
Since a proper map between manifolds is a closed map \cite{MR254818}, the following theorem follows from \Cref{degreeonemapchecking}.
\begin{theorem} \label{non-surjectivepropermaphasdegreezero}
    \textup{Let $f\colon M\to N$ be a proper map between two connected, oriented, smooth $n$-manifolds such that $f(\partial M)\subseteq \partial N$. If $\deg(f)\neq 0$, then $f$ is surjective.}
\end{theorem}

\section{Proofs of the main theorems}
We begin with the proof of  \Cref{all}.
\begin{proof}[Proof of \Cref{all}]
    Let $p\colon \widetilde \Sigma\to \Sigma$ be the covering corresponding to the subgroup $\operatorname{im}\pi_1(f)$ of $\pi_1(\Sigma)$, and let $\widetilde f\colon \Sigma'\to \widetilde \Sigma$ be a lift of $f$ with respect to $p$, i.e., $p\circ \widetilde f=f$. Thus, $\operatorname{im}\pi_1(p)=\operatorname{im}\pi_1(f)$, and hence $\pi_1(\widetilde f)\colon \pi_1(\Sigma')\to \pi_1(\widetilde \Sigma)$ is an isomorphism because a covering map induces injection between fundamental groups. Note that if a non-compact surface has an infinite cyclic (resp. trivial) fundamental group, then it is homeomorphic to the punctured plane (resp. plane). Also, the properness of $f$ implies the properness of $\widetilde f$. Since non-compact surfaces are $K(\pi,1)$ CW-complexes, by Whitehead theorem $\widetilde f$ is a homotopy equivalence. Thus, by \cite[Theorem 1]{MR4843735}, $\widetilde f$ is properly homotopic to a homeomorphism, which implies $\deg (\widetilde f)=\pm 1$.
    
    We claim that $p$ is a $d$-sheeted covering for some positive integer $d$. Notice that it is sufficient to show that $p$ is a proper map. On the contrary, suppose $p$ is not proper, i.e., $p^{-1}(x)$ is infinite for some $x\in \Sigma$.  Define $C'\coloneqq f^{-1}(x)\subseteq \Sigma'$. So, the compact set $\widetilde f(C')$ is contained in the discrete space $p^{-1}(x)$, i.e., we can write $\widetilde f(C')=\left\{\widetilde x_1,...,\widetilde x_k\right\}$ for points $\widetilde x_1,...,\widetilde x_k\in p^{-1}(x)\subset\widetilde \Sigma$. Since $p^{-1}(x)$ is infinite, we have an $\widetilde x\in p^{-1}(x)\setminus \left\{\widetilde x_1,...,\widetilde x_k\right\}$. Thus, $\widetilde f(\Sigma')\subseteq \widetilde \Sigma\setminus \widetilde x$. Since $\widetilde \Sigma$ is orientable,  $H^n(\widetilde \Sigma, \widetilde \Sigma\setminus \widetilde y)\to H^n_c(\widetilde \Sigma)$ is an isomorphism for every $\widetilde y\in \widetilde \Sigma$. Now, the following commutative diagram shows that $\deg(\widetilde f)=0$, which contradicts the fact that $\deg(\widetilde f)\neq 0$. Hence, $p$ must be a proper map. $$\begin{tikzcd}
{ H^n_c(\widetilde \Sigma)} \arrow[rr, " \widetilde f^*"]                                             &  & { H^n_c(\Sigma')} \\
{H^n(\widetilde \Sigma, \widetilde \Sigma\setminus \widetilde x)} \arrow[rr, " \widetilde f^*"'] \arrow[u, "\cong"] &  & { H^n(\Sigma',\Sigma')=0} \arrow[u]      
\end{tikzcd}$$ 
Fix an orientation of $\widetilde \Sigma$. By \Cref{degreeonemapchecking}, $\deg(p)=\pm d$. Since $\deg(f)=\deg(p\widetilde f)=(\pm d)\cdot \deg(\widetilde f)=(\pm d)\cdot (\pm 1)$, we can conclude that $\deg(f)=\pm d(\neq 0)$, and $f$ is properly homotopic to the $d$-sheeted covering map $p.$\end{proof}

The remainder of this article is based on \emph{surgery on a proper map between non-compact surfaces}. We will first recall a few terminologies and theorems.

Let $\mathrm S$ be a connected, orientable two-dimensional manifold with or without boundary. A \emph{circle} on $\mathrm S$ is the image of an embedding of $\mathbb  S^1$ into $\mathrm S$. We say a circle $\mathcal C$ on $\mathrm S$ is a \emph{trivial circle} if $\mathcal C$ bounds an embedded disk in $\mathrm S$. A circle on $\mathrm S$ will be called \emph{primitive} if it is not trivial. Note that if the image of an embedding $f\colon \mathbb  S^1\hookrightarrow \mathrm S$ is a primitive circle on $\mathrm S$, then $[f]\in \pi_1(\mathrm S)$ is a primitive element \cite[Theorems 1.7 and  4.2]{MR214087}. A pairwise disjoint collection $\mathscr A=\{\mathcal C_\alpha\}$ of smoothly embedded circles on $\mathrm S$ is called a \emph{locally finite curve system} (in short, \textup{LFCS}) on $\mathrm S$ if for each compact subset $K$ of $\Sigma$, $\mathcal C_\alpha\cap K=\varnothing$ for all but finitely many $\alpha$. The theorem below tells that the transversal pre-image of an \text{LFCS} is again an \text{LFCS}. 
\begin{theorem}[Approximation and transversality in proper category]
    \textup{\cite[$\S3.2$]{MR4843735} Let $f\colon \Sigma'\to \Sigma$ be a proper map between non-compact surfaces, and let $\mathscr A$ be an \textup{LFCS} on $\Sigma$. Then $f$ can be properly homotoped so that it becomes smooth as well as transverse to $\mathscr A$. Thus, after such a proper homotopy, for each component $\mathcal C$ of $\mathscr A$, either $f^{-1}(\mathcal C)$ is empty or a pairwise disjoint finite collection of smoothly embedded circles on $\Sigma'$; in particular, $f^{-1}(\mathscr A)$ is an \textup{LFCS} on $\Sigma'$.}\label{whiteny}
\end{theorem}
Let $\Sigma$ be a non-compact surface. Suppose $\mathscr A$ is an \textup{LFCS} on $\Sigma$ and $\{\Sigma_n\}$ is an at most countable collection of bordered subsurfaces of $\Sigma$. We say $\mathscr A$ \emph{decomposes $\Sigma$ into bordered subsurfaces}, where \emph{complementary components} are  $\{\Sigma_n\}$ if  each $\Sigma_n$ is a closed subset of $\Sigma$, $\operatorname{int}(\Sigma_n)\cap \operatorname{int}(\Sigma_m)=\varnothing$ if $n\neq m$, $\cup_n\Sigma_n=\Sigma$, and $\cup_n\  \partial\Sigma_n=\cup \mathscr A$.  For example, if $\Sigma\neq \mathbb  R^2$, then there is an \textup{LFCS} $\mathscr C$ on $\Sigma$ such that $\mathscr C$ decomposes $\Sigma$ into bordered subsurfaces, and a complementary component of this decomposition is either a two-hold torus, a pair of pants, or a punctured disk. This follows from Goldman’s inductive procedure \cite[Section 2.6]{MR275436}. We call an \textup{LFCS} $\mathscr A$ on $\Sigma$  a \emph{preferred \textup{LFCS}} if it satisfies either of the following: \begin{enumerate}
    \item $\mathscr A$ is a finite collection of primitive circles on $\Sigma$. 
    \item $\mathscr A$ decomposes $\Sigma$ into bordered subsurfaces, and a complementary component of this decomposition is either a two-holed torus, a pair of pants, an annulus, or a punctured disk.
\end{enumerate} Now, we state our first theorem on surgery on a proper map. 

\begin{theorem}[Disk removal]\textup{\cite[Theorem 3.3.5]{MR4843735} Let $f\colon \Sigma'\to \Sigma$ be a smooth proper map between two non-compact surfaces, where $\Sigma'\neq \mathbb  R^2\neq \Sigma$. Suppose $\mathscr A$ is a preferred \textup{LFCS} on $\Sigma$ such that $f$ is transverse to $\mathscr A$. Then, $f$ can be properly homotoped to remove all trivial components of the \textup{LFCS} $f^{-1}(\mathscr A)$ such that each primitive component of $f^{-1}(\mathscr A)$ has an open neighborhood on which this proper homotopy is constant.} \label{diskremoval}\end{theorem}

\begin{remark}
    \textup{The only usage of the hypothesis $\Sigma'\neq \mathbb  R^2$ to prove \Cref{diskremoval} is to discard the possibility that there exist infinitely many components $\mathcal C_1',\mathcal C_2',...$ of $f^{-1}(\mathscr A)$ bounding the disks $\mathcal D_1',\mathcal D_2',...$, respectively such that $\mathcal C_n'$ is contained in the interior of $\mathcal D_{n+1}'$ for each $n$. Thus, if $\mathscr A$ has only finitely many components, then the same conclusion of \Cref{diskremoval} holds even if we consider $\mathbb  R^2$ as the domain of $f$. However, note that the hypothesis $\Sigma\neq \mathbb  R^2$ can never be dropped because to send every disk in $\Sigma'$ bounded by a component of $f^{-1}(\mathscr A)$ into a circle $\mathcal C_\delta\subset \Sigma\setminus \cup\mathscr A$ using homotopy long exact sequence, $\mathcal C_\delta$ must be a primitive circle parallel to a component of $\mathscr A$.}\label{diskremovalremark}
\end{remark}
At this point, we are ready to prove the easy part of \Cref{plane}.

\begin{proof}[Proof of part~\ref{plane1} of \Cref{plane}] First, assume $\Sigma=\mathbb  C$. By \Cref{reverseprocessindegreefinding}, properly homotope $f$ so that the image of $f$ misses a point $a\in \mathbb  C$. Since any translation map of $\mathbb  C$ is properly homotopic to the identity map of $\mathbb  C$, we may assume $a=0$. Thus, there exists $r>0$ such that $|f|\geq r$ because proper maps between manifolds are closed maps. Consider a compact subset $K$ of $\mathbb  C.$ Let $n$ be a positive integer such that $K\subseteq \{z\in \mathbb  C:|z|\leq nr\}$. Since $1+nt\geq 1$ for every $t\in [0,1]$, the map $\mathbb  C\times [0,1]\ni (z,t)\longmapsto (1+nt)\cdot f(z)\in \mathbb  C$ is a proper homotopy from $f$ to $g\coloneqq(n+1)f$. Certainly, $\operatorname{im}(g)\cap K=\varnothing$. So, we are done when $\Sigma=\mathbb  C.$

Now, suppose $\Sigma\neq \mathbb  C$. Let $K$ be a compact subset of $\Sigma$. Using Goldman’s inductive procedure \cite[Section 2.6]{MR275436}, find a compact bordered subsurface $\mathrm S$ of $\Sigma$ such that $K\subseteq \operatorname{int}(\mathrm S)$ and each component of $\partial \mathrm S$ is a primitive circle on $\Sigma$. By \Cref{whiteny}, we may assume $f$ is smooth as well as transverse to $\partial \mathrm S$. Therefore, $f^{-1}(\partial \mathrm S)$ is a pairwise disjoint finite collection of smoothly embedded circles on $\mathbb  C$. Since $f^{-1}(\partial \mathrm S)$ has finitely many components, each of which bounds a disk in $\mathbb  C$, by considering a procedure similar to that given in proof of \Cref{diskremoval}, we can properly homotope $f$ to a proper map $g\colon \mathbb  C\to \Sigma$ so that $g^{-1}(\partial \mathrm S)=\varnothing$ (see \Cref{diskremovalremark}). Using continuity of $g$, either $\operatorname{im}(g)\subset \mathrm S$ or $\operatorname{im}(g)\subseteq \Sigma\setminus \mathrm S$. Since $g$ is proper, the former can't happen, i.e., $\operatorname{im}(g)\subseteq \Sigma\setminus \mathrm S$, and hence $\operatorname{im}(g)\cap K=\varnothing$.
\end{proof}
Notice that, unlike the first part, in the second part of the above proof, we have not utilized the
geometric realization of the fact that \( \deg(f) = 0 \). This is because if \( f \) is a proper map from \( \mathbb{C} \) to a non-compact oriented surface \( \Sigma \) of non-zero degree, then \( \Sigma \) must be \( \mathbb{C} \), by \Cref{degreeonemapsarepi1surjective}.

Our next goal is to understand the number of proper homotopy classes of degree-zero maps from \( \mathbb{C} \) to an oriented surface \( \Sigma \). To do so, we will need the following lemma, whose proof is obtained by modifying the Alexander trick \cite[Lemma 2.1]{MR2850125}.

\begin{lemma}\textup{Let $\mathbb  B^2_*\coloneqq \{z\in \mathbb  C:0<|z|\leq 1\}$. Suppose $f\colon \mathbb  B^2_*\to \mathbb  B^2_*$ is a proper map. Then $f$ can be properly homotoped relative to  $\mathbb  S^1$ to a proper map $g\colon \mathbb  B^2_*\to \mathbb  B^2_*$ such that $g(z)=|z|\cdot f(z/|z|)$ for each $z\in\mathbb  B^2_*$.}\label{Alexandertrick2}\end{lemma}
\begin{proof}
  Define  a proper map $\mathcal H\colon \mathbb  B^2_*\times [0,1]\to \mathbb  B^2_*$  by $$\mathcal H(z,t)\coloneqq \begin{cases}(1-t)\cdot f(z/(1-t)) & \text{ if }0<|z|\leq 1-t,\\  |z|\cdot f(z/|z|) & \text{ if }1-t<|z|\leq 1.
      
  \end{cases}$$ Let $g\coloneqq \mathcal H(-,1)$. Thus $\mathcal H$ is a proper homotopy from $f$ to $g$, relative to $\mathbb S^1$.
\end{proof}

\begin{theorem}
    \textup{Let $\Sigma$ be an oriented non-compact surface, and let $e$ be an end of $\Sigma$. Denote by \( [\mathbb C, \Sigma]_e \) the set of proper homotopy classes of degree zero maps \( f\colon \mathbb{C} \to \Sigma \) such that \( \operatorname{Ends}(f) \) sends the unique end of \( \mathbb{C} \) to \( e \). If \( e \) is an isolated planar end of \( \Sigma \), then \( [\mathbb C, \Sigma]_e \) is a singleton. In contrast, if \( e \) is either non-isolated or non-planar, then \( [\mathbb C, \Sigma]_e \) is uncountable.}\label{ref}
\end{theorem}
\begin{proof}Suppose \( e \) is an isolated planar end of \( \Sigma \). Let \( f \) and \( g \) be two degree-zero maps from \( \mathbb{C} \) to \( \Sigma \) such that both \( \operatorname{Ends}(f) \) and \( \operatorname{Ends}(g) \) send the unique end of \( \mathbb{C} \) to \( e \). We need to show that \( f \) and \( g \) are properly homotopic. For this, it suffices to show that, after proper homotopies, \( f = g \). By part~\ref{plane1} of \Cref{plane}, after applying proper homotopies, we may regard \( f \) and \( g \) as maps from \( \mathbb{C} \) to \( \mathbb{B}^2_* := \{ z \in \mathbb{C} : 0 < |z| \leq 1 \} \). Choose a disk \( \mathrm{D} \subset \mathbb{C} \). Since both \( f|\mathrm{D} \) and \( g|\mathrm{D} \) are homotopic to the constant map \( \mathrm{D} \to \{1\} \subset \mathbb{B}^2_* \), the proper homotopy extension theorem \cite[Theorem 10.1.11]{MR2365352} allows us to further properly homotope \( f \) and \( g \) so that \( f(\mathrm{D}) = \{1\} = g(\mathrm{D}) \). Now, since \( \mathrm D^-\coloneqq\mathbb{C} \setminus \operatorname{int}(\mathrm{D}) \) is homeomorphic to \( \mathbb{B}^2_* \), we can apply \Cref{Alexandertrick2} to  \(f|\mathrm D^- \to \mathbb{B}^2_*\) and \( g|\mathrm D^-  \to \mathbb{B}^2_*\) to conclude that, after proper homotopies relative to $\partial \mathrm D^-=\partial \mathrm D$, \( f(z) = g(z) \in (0,1] \) for all \( z \in \mathrm D^- \).  By pasting these with \( f|\mathrm{D}= g|\mathrm{D}\), we are done.

Now, we consider the second case. Suppose \( e \) is not an isolated planar end of \( \Sigma \). In particular, \( \Sigma \) cannot be homeomorphic to \( \mathbb{C} \). Let \( \mathcal{SE}(\Sigma, e) \), called the \emph{set of strong ends determined by} \( e \), denote the set of proper homotopy classes of proper maps \( r\colon [0,\infty) \to \Sigma \) such that \( r \) represents the equivalence class \( e \). There is a well-defined map \( \Theta\colon [\mathbb{C}, \Sigma]_e \to \mathcal{SE}(\Sigma, e) \) given by restricting a degree-zero map \( f\colon \mathbb{C} \to \Sigma \) to the subspace \( [0,\infty) \subset \mathbb{C} \). This map \( \Theta \) is surjective: if \( r\colon [0,\infty) \to \Sigma \) is a proper map, then the map \( f_r\colon \mathbb{C} \ni z \longmapsto r(|z|) \in \Sigma \) is a proper extension of \( r \), and \( \deg(f_r) = 0 \). This follows from the fact that any proper map \( \mathbb{C} \to \Sigma \), with \( \Sigma \neq\mathbb{C} \), must have degree zero by \Cref{degreeonemapsarepi1surjective}. Thus, to complete the proof, it suffices to show that \( \mathcal{SE}(\Sigma, e) \) is uncountable.

 Fix a representative \(\omega\colon[0,\infty)\to \Sigma\) of \(e\). Let \(\{K_n\colon n\geq 0\}\) be a sequence of compact bordered subsurfaces of \(\Sigma\) such that \(\bigcup_n K_n = \Sigma\) and \(K_n \subset \operatorname{int}(K_{n+1})\) for each \(n \geq 0\). Moreover, for each \( n \geq 0 \), we may assume that the closure \( C \) of every component of \( \Sigma \setminus K_n \) is non-compact, and that \( C \cap K_n \) is a component of \( \partial K_n \). Such a sequence \(\{K_n:n\geq 0\}\) can be constructed, for example, via Goldman's inductive procedure. Without loss of generality, we may also assume that \(\omega([n,\infty))\) is contained in a component—say \(U_n\)—of \(\Sigma \setminus K_n\) for each \(n \geq 0\). Thus, each \(\underline{U_n}\) is a neighborhood of \(e\), and \(U_{n+1} \subset U_n\) for all \(n \geq 0\), and $\bigcap_n U_n=\varnothing$. Consider the inverse system
\begin{equation}\label{eq}
    \cdots \xrightarrow{\varphi_2} \pi_1(U_2, \omega(2)) \xrightarrow{\varphi_1} \pi_1(U_1, \omega(1)) \xrightarrow{\varphi_0} \pi_1(U_0, \omega(0))
\end{equation}
of groups, where each bonding homomorphism \(\varphi_n\) is defined by the inclusion \(U_{n+1} \hookrightarrow U_n\), together with a change of basepoint along the path \(\omega|[n,n+1]\). By  \cite[Proposition 16.1.1]{MR2365352}, there is a canonical bijection from \(\mathcal{SE}(\Sigma,e)\) to the first derived limit \(\varprojlim^1 \pi_1(U_n, \omega(n))\). Now, note that if the first derived limit of an inverse system \( \{G_n\} \), where each \( G_n \) is at most countable, is at most countable, then the system must satisfy the Mittag--Leffler condition \cite[Lemma 3.3]{MR2553079}. Therefore, to prove that \( \mathcal{SE}(\Sigma, e) \) is uncountable, it suffices to show that \eqref{eq} does not satisfy the Mittag--Leffler condition. To see this, first note that for each \( n \geq 0 \), since \( \pi_1(U_n) \cong \pi_1(U_{n+1} \cup \partial U_{n+1}) *_{\pi_1(\partial U_{n+1})} \pi_1(U_n \setminus U_{n+1}) \) and \( \partial U_{n+1} \) is a primitive circle in \( U_n \), it follows that the inclusion \( U_{n+1} \cup \partial U_{n+1} \hookrightarrow U_n \) is \( \pi_1 \)-injective. Hence, the inclusion \( U_{n+1} \hookrightarrow U_n \) is also \( \pi_1 \)-injective for all \( n \geq 0 \). Therefore, each \( \varphi_n\), being the composition of the inclusion-induced monomorphism \( \pi_1(U_{n+1}, \omega(n+1)) \to \pi_1(U_n, \omega(n+1)) \) and a basepoint-changing isomorphism, is injective. Now, since \( e \) is either non-isolated or non-planar, for every \( n \geq 0 \), there exists \( m_n \geq n+1 \) such that the closure (taken in \( \Sigma \)) of \( U_n \setminus U_{m_n} \) either has more than one end or contains more than one essential one-holed torus. Passing to a subsequence if necessary, we henceforth assume that \( m_n = n+1 \) for each \( n \geq 0 \). Thus, the inclusion-induced map \( \pi_1(U_{n+1}, \omega(n+1)) \to \pi_1(U_n, \omega(n+1)) \) is non-surjective for each \( n \geq 0 \), and hence \( \operatorname{im}(\varphi_n) \) is a proper subgroup of \( \pi_1(U_n, \omega(n)) \) for all \( n \geq 0 \). In particular, $\operatorname{im}(\varphi_0\circ \cdots \circ \varphi_n)\neq \operatorname{im}(\varphi_0\circ \cdots \circ \varphi_{n+1})$ for all $n\geq 0$. Therefore, \eqref{eq} does not satisfy the Mittag--Leffler condition. 
\end{proof}

Another surgery technique is needed to complete the proof of \Cref{plane}.

\begin{theorem}[Fiber-preserving homeomorphisms of tubular neighborhoods via local homotopy]\label{deg1tohomeo}
\textup{\cite[Theorem 3.4.3]{MR4843735} Let $f\colon \Sigma'\to \Sigma$ be a smooth proper map between two non-compact surfaces, where $\Sigma'\neq \mathbb  R^2\neq \Sigma$. Suppose $\mathscr A$ is a preferred \textup{LFCS} on $\Sigma$ such that $f$ is transverse to $\mathscr A$. If $f$ is a homotopy equivalence, then $f$ can be properly homotoped to remove all trivial components of the $f^{-1}(\mathscr A)$ as well as to map each primitive component of $f^{-1}(\mathscr A)$ homeomorphically onto a component of $\mathscr A$. Moreover, after this proper homotopy, if $\mathcal C_{\textbf{\textup{p}}}'$ and $\mathcal C$ are components of $f^{-1}(\mathscr A)$ and $\mathscr A$, respectively, such that  $f\vert \mathcal C_{\textbf{\textup{p}}}'\to \mathcal C$ is a homeomorphism, then $\mathcal C_{\textbf{\textup{p}}}'$ (resp. $\mathcal C$)  has two one-sided tubular neighborhoods $\mathcal U'$ and $\mathcal V'$ (resp. $\mathcal U$ and $\mathcal V$) with some identifications $(\mathcal U',\mathcal C_{\textbf{\textup{p}}}')\equiv(\mathcal C'_\textbf{p}\times [1,2],\mathcal C_{\textbf{\textup{p}}}'\times \{2\})\equiv (\mathcal V',\mathcal C_{\textbf{\textup{p}}}')$ \big(resp. $(\mathcal U,\mathcal C)\equiv (\mathcal C\times [1,2],\mathcal C\times \{2\})\equiv (\mathcal V,\mathcal C)$\big) such that the following hold:\begin{enumerate}
    \item $\mathcal U'\cup \mathcal V'$ is a (two-sided) tubular neighborhood of $\mathcal C_{\textbf{\textup{p}}}'$;
    \item $f\vert \mathcal U'\to \mathcal U$ and $f\vert \mathcal V'\to \mathcal V$ are homeomorphisms described by $\mathcal C_{\textbf{\textup{p}}}'\times [1,2]\ni (z,t)\longmapsto \big(f(z),t\big)\in \mathcal C\times [1,2].$
\end{enumerate}}
\end{theorem}
\begin{remark}
    \textup{Instead of assuming that $f$ is a homotopy equivalence in \Cref{deg1tohomeo}, if we assume that for a primitive component $\mathcal{C}_\textbf{p}'$ of $f^{-1}(\mathscr{A})$, there exists a component $\mathcal{C}$ of $\mathscr{A}$ such that $f(\mathcal C_\textbf{p}')\subset \mathcal C$ and the restriction $f\vert \mathcal{C}_\textbf{p}' \to \mathcal{C}$ has degree $n \neq 0$, then $f$ can be properly homotoped so that $f\vert \mathcal{C}_\textbf{p}' \to \mathcal{C}$ becomes an $n$-sheeted covering map. Moreover, after this proper homotopy, $f$ maps each small enough one-sided tubular neighborhood $\mathcal C_{\textbf{\textup{p}}}'\times [1,2],\ \mathcal C_\textbf{p}'\times \{2\}\equiv\mathcal C_\textbf{p}'$ of $\mathcal C_\textbf{p}'$ (on either side of $\mathcal C_\textbf{p}'$) onto a small enough one-sided tubular neighborhood $\mathcal C\times [1,2],\ \mathcal C\times \{2\}\equiv\mathcal C$ of $\mathcal C$ in the following manner: $\mathcal C_{\textbf{\textup{p}}}'\times [1,2]\ni (z,t)\longmapsto \big(f(z),t\big)\in \mathcal C\times [1,2].$}\label{deg1tohomeoremark}
\end{remark}
\begin{remark}
    \textup{Instead of assuming that $f$ is a homotopy equivalence in \Cref{deg1tohomeo}, if we assume that $f$ is $\pi_1$-injective, then $f$ can be properly homotoped to remove all trivial components from $f^{-1}(\mathscr A)$ as well as to send each component $\mathcal{C}_\textbf{p}'$ of $f^{-1}(\mathscr{A})$ onto a component $\mathcal{C}$ of $\mathscr{A}$ such that the restriction $f\vert \mathcal{C}_\textbf{p}' \to \mathcal{C}$ is a covering map. Moreover, after this proper homotopy, $f$ maps each small enough one-sided tubular neighborhood $\mathcal C_{\textbf{\textup{p}}}'\times [1,2],\ \mathcal C_\textbf{p}'\times \{2\}\equiv\mathcal C_\textbf{p}'$ of $\mathcal C_\textbf{p}'$ (on either side of $\mathcal C_\textbf{p}'$) onto a small enough one-sided tubular neighborhood $\mathcal C\times [1,2],\ \mathcal C\times \{2\}\equiv\mathcal C$ of $\mathcal C$ in the following manner: $\mathcal C_{\textbf{\textup{p}}}'\times [1,2]\ni (z,t)\longmapsto \big(f(z),t\big)\in \mathcal C\times [1,2].$}\label{deg1tohomeoremark2}
\end{remark}
Note that the proofs of \Cref{deg1tohomeo}, \Cref{deg1tohomeoremark}, and \Cref{deg1tohomeoremark2} all rely on applying the following fact within small enough one-sided tubular neighborhoods of primitive circles in a fiber-preserving manner: up to homotopy, any map \( \mathbb{S}^1 \to \mathbb{S}^1 \) is of the form \( z \longmapsto z^n \) for some integer \( n \). Consequently, these proof techniques extend naturally to the setting of bordered surfaces, where the proper homotopy may also be taken relative to the boundary, provided the tubular neighborhoods are disjoint from the boundaries of both the domain and the codomain.

The proof of part~\ref{plane2} of \Cref{plane} also requires the following straightforward lemmas.

\begin{lemma}
    \textup{If $\theta\colon \mathbb  S^1\times [0,1]\to \mathbb  S^1\times [0,1]$ is a map which sends $\mathbb  S^1\times \{0,1\}$ into $\mathbb  S^1\times \{0\}$, then $\varphi$ can be homotoped relative to  $\mathbb  S^1\times \{0,1\}$ to send $\mathbb  S^1\times [0,1]$ into $\mathbb  S^1\times \{0\}$. } \label{particularannuluscompression}
\end{lemma}
\begin{lemma}\label{pushingleft}
\textup{Let $\varphi\colon \mathbb  S^1\times [1,3]\to \mathbb  S^1\times [1,2]$ be a map such that $\varphi(\mathbb  S^1\times \{r\})\subseteq \mathbb  S^1\times \{r\}$  for $1\leq r\leq 2$ and $\varphi\left(\mathbb  S^1\times [2,3]\right)\subseteq \mathbb  S^1\times \{2\}$. Then $\varphi$ can be homotoped relative to   $\mathbb  S^1\times \{1,3\}$ to satisfy $\varphi^{-1}(\mathbb  S^1\times \{2\})=\mathbb  S^1\times \{3\}$.} \end{lemma}
So we now prove part~\ref{plane2} of \Cref{plane}.

\begin{proof}[Proof of part~\ref{plane2} of \Cref{plane}] By \Cref{degreeonemapsarepi1surjective}, $\Sigma=\mathbb  C$. At first, suppose $n>0$. Let $\mathscr{A}$ be the collection of all circles in $\mathbb{C}$ centered at $0$ with integer radii. Then $\mathscr A$ is an \textup{LFCS} on $\mathbb  C$. By \Cref{whiteny}, we may assume $f$ is a smooth proper map of degree $n>0$ as well as transverse to $\mathscr A$. Therefore, $f^{-1}(\mathscr A)$ is an \textup{LFCS} by \Cref{whiteny}. Now, \Cref{reverseprocessindegreefinding} gives a proper map $g$ properly homotopic to $f$ with the following properties: the map $g$ equals $f$ outside a compact subset $K$ of the form $\{z\in \mathbb  C:|z|\leq r\}$ for some $r>0$, and there exists a disk $D$ in $\mathbb  C$ such that $g^{-1}(D)$ is the union of pairwise disjoint disks $D_1,..., D_n$ in $\mathbb  C$ such that $g\vert D_j \to D$ is an orientation-preserving homeomorphism for each $j=1,...,n$. Without loss of generality, we may assume $K$ contains $\cup_{j=1}^n D_j$. Consider the bordered surfaces $\mathbf S'\coloneqq \mathbb  C\setminus \cup_{j=1}^n \operatorname{int}(D_k)$ and $\mathbf S\coloneqq \mathbb  C\setminus \operatorname{int}(D)$. Notice that $g(\mathbf S')\subseteq \mathbf S$ and $g$ sends each component of  $\partial \mathbf S'=g^{-1}(\partial \mathbf S)$ homeomorphically onto $\partial \mathbf S$. Denote the restriction map $g\vert \mathbf S'\to \mathbf S$ by $g_\text{res}$. Since $\mathscr A$ is an \textup{LFCS}, there exists a component $\mathcal C$ of $\mathscr A$ such that $\mathcal C\cap \big(f(K)\cup g(K)\big)=\varnothing$ and the interior of $\mathcal C$ (in $\mathbb  C$) contains $D$. Since $f=g$ on $\mathbb  C\setminus K$, we can say that $f^{-1}(\mathcal C)=g^{-1}(\mathcal C)=g^{-1}_\text{res}(\mathcal C)$ doesn't intersect with $K$, $g_\text{res}$ is smooth outside the compact subset $\mathbf K'\coloneqq \mathbf S'\cap K$ of $\mathbf S'$, and $g_\text{res}$ is transverse to $\mathcal C$. In particular, each component of $g_\text{res}^{-1}(\mathcal C)$ is either a trivial circle or primitive circle on $\mathbf S'$ lying in $\mathbf S'\setminus \mathbf K'=\mathbb  C\setminus K=\{z\in \mathbb  C:|z|>r\}$. 
\begin{figure}
\centering    
\includegraphics[width=1\textwidth]{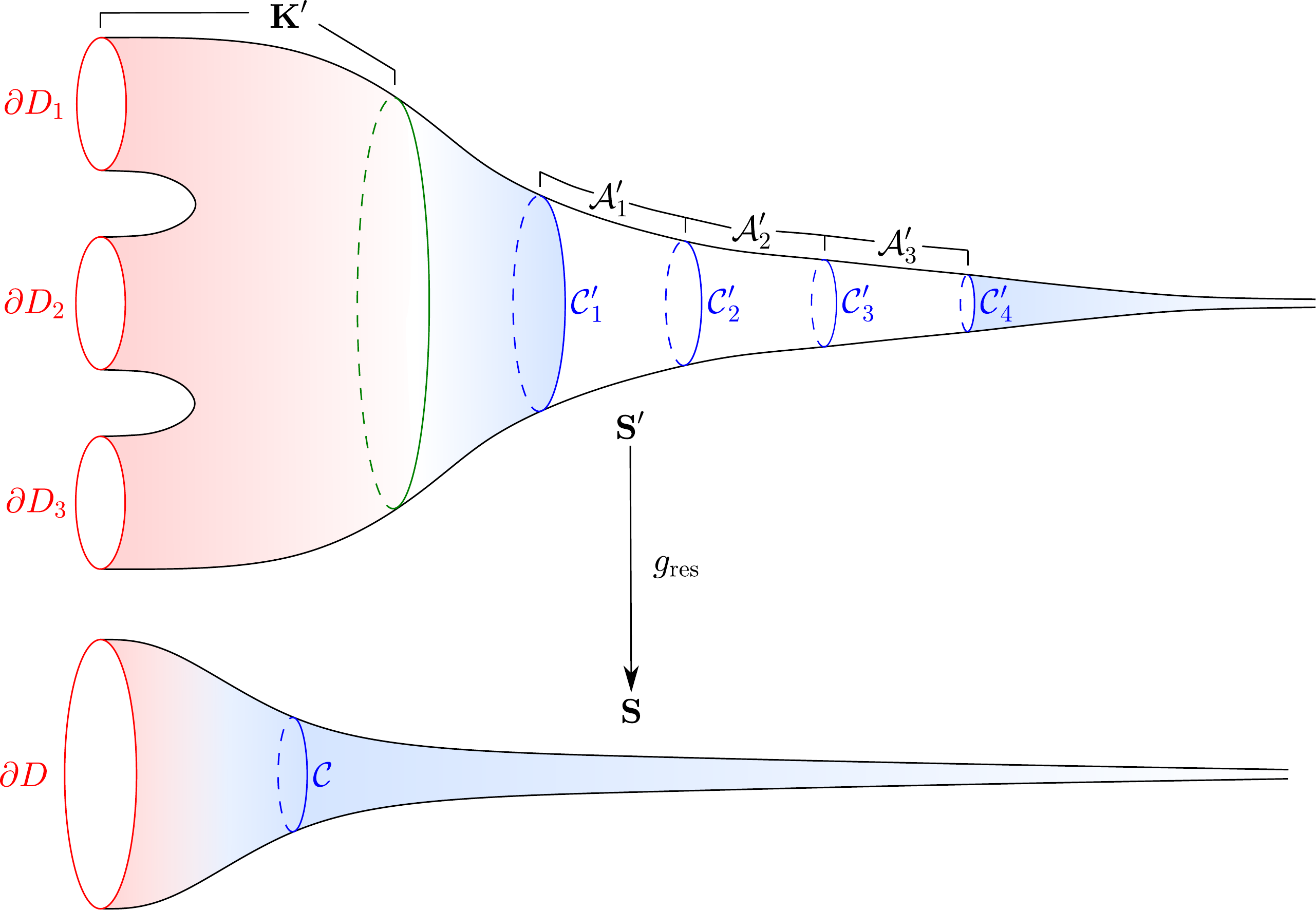}
\caption[Classification of non-zero self-maps of the plane]{The description of $g_\text{res}\colon \mathbf S'\to \mathbf S$ after removing trivial components from $g_\text{res}^{-1}(\mathcal C)$.}\label{R2}
\end{figure}

Now, we aim to properly homotope $g_\text{res}\colon \mathbf S'\to \mathbf S$ to a proper map $\widetilde{g_\text{res}}\colon\mathbf S'\to \mathbf S$ such that $\mathcal C'\coloneqq\widetilde{g_\text{res}}^{-1}(\mathcal C)$ is a single circle in $\operatorname{int}(\mathbf S')$ and $\widetilde{g_\text{res}}\vert \mathcal C'\to \mathcal C$ is an $n$-sheeted covering map, but here we want every proper homotopy to be relative to $\partial \mathbf S'$.  So, here are the steps.
    
    At first, applying \Cref{reverseprocessindegreefinding} on $g_\text{res}$, and then using \Cref{degreeonemapchecking}, we can tell $\deg\left(g_\text{res}\right)=\deg(g)$, i.e., in particular, $\deg\left(g_\text{res}\right)\neq 0$. Recall that a proper homotopy relative to the boundary doesn't change the degree. Therefore, $g_\text{res}$ remains surjective after any proper homotopy relative to  $\partial \mathbf S'$; see \Cref{non-surjectivepropermaphasdegreezero}.

    Since $\mathcal C$ is a primitive circle on $\mathbf S$, and the proof of \Cref{diskremoval} adapts readily to the current setting of bordered surfaces, after a proper homotopy $\mathcal H_1$, we may assume each component of the non-empty set $g^{-1}_\text{res}(\mathcal C)$ is a primitive circle on $\mathbf S'$. Note that the homotopy that appears in the proof of \Cref{diskremoval} is, in fact, relative to the complement of every small enough neighborhood of the union of all disks bounded by trivial circles. Thus, we may assume $\mathcal{H}_1$ is relative to $\partial \mathbf{S'}$.

   Therefore, for some positive integer $m$, we can write $g^{-1}_\text{res}(\mathcal C)=\mathcal C_1'\sqcup \cdots \sqcup\mathcal C_m'$ (see \Cref{R2} for labeling of these circles), where each $\mathcal C_i'$ is a primitive circle on $\mathbf S'$. Thus $g^{-1}_\text{res}(\mathcal C)$ decomposes $\mathbf S'$ into a copy of $S_{0,n+1}$, $(m-1)$-copies $\mathcal A_1',...,\mathcal A_{m-1}'$ of $S_{0,2}$ (see \Cref{R2} for labeling of these annuli), and a copy of $S_{0,1,1}$. Similarly, $\mathcal C$ decomposes $\mathbf S$ into a copy of $S_{0,2}$ and a copy of $S_{0,1,1}$.

   Restricting $g_\text{res}$ on those copies, we get a map $\xi\colon S_{0,n+1}\to S_{0,2}$ such that $\xi\vert \partial D_j\to \partial D$ is an orientation-preserving homeomorphism for each $j=1,...,n$ and $\xi$ sends the other boundary component $\partial S_{0,n+1}\setminus \cup_{j=1}^n \partial D_j$ into $\mathcal C$. The naturality of the homology long exact sequence and \cite[Exercise 31 of Section 3.3]{MR1867354} give the following commutative diagram 
$$\begin{tikzcd}
{H_2\left(S_{0,n+1}, \partial S_{0,n+1}\right)\cong \mathbb  Z} \arrow[d, "\times \deg(\xi)"'] \arrow[rr, "1\longmapsto \overset{n+1}{\oplus} 1"] &  & { \overset{n+1}{\oplus}\mathbb  Z\cong H_1\left(\partial S_{0,n+1}\right)} \arrow[d, "\overset{n+1}{\oplus} 1\longmapsto n\oplus \ell"] \\
{H_2\left(S_{0,2}, \partial S_{0,2}\right)\cong \mathbb  Z} \arrow[rr, " 1\longmapsto \overset{2}{\oplus}1"']                                     &  & { \overset{2}{\oplus}\mathbb  Z\cong H_1\left(\partial  S_{0,2}\right)}  \end{tikzcd}$$ where $\ell\coloneqq\deg\left(\xi\vert \partial S_{0,n+1}\setminus \cup_{j=1}^n \partial D_j\to \mathcal C\right)$. The commutativity of the diagram tells $\deg(\xi)=n$ and $\ell=n$. Since any two components of $g_\text{res}^{-1}(\mathcal C)$ co-bound an annulus in $\mathbf S'$ and any two homotopic maps $\mathbb  S^1\to \mathbb  S^1$ have the same degree, $\deg\left(g_\text{res}\vert \mathcal C_i'\to \mathcal C\right)=\pm n$ for each $i=1,...,m$ (the minus sign comes because two boundary components of an oriented annulus are oppositely oriented). Thus, since the proof of \Cref{deg1tohomeo}, together with \Cref{deg1tohomeoremark}, adapts readily to the current setting of bordered surfaces, after a proper homotopy $\mathcal H_2$, we may assume that $g_\text{res}\vert \mathcal C_i'\to \mathcal C$ is an $n$-sheeted covering map for each $i=1,...,m$. Since the homotopy mentioned in the proof of \Cref{deg1tohomeo} is relative to the complement of a small neighborhood of $g_\text{res}^{-1}(\mathcal C)=\mathcal C_1'\sqcup \cdots \sqcup \mathcal C_m'$, one can ensure that $\mathcal{H}_2$ is relative to $\partial \mathbf{S'}$.

Now, we remove the components $\mathcal C_2',..., \mathcal C_m'$ from  $g_\text{res}^{-1}(\mathcal C)$. First, consider the annulus $\mathcal A_1'$. By continuity, $g_\text{res}(\mathcal A_1')$ will be one of two sides of $\mathcal C$ in $\mathbf S$, i.e., the compact set $g_\text{res}(\mathcal A_1')$ is contained in a one-sided tubular neighborhood of $\mathcal C$ in $\mathbf S$. Notice that if $\mathbf S''$ is a compact bordered subsurface of $\mathbf S'$ such that $\mathbf S''\cap \partial \mathbf S'=\varnothing$, then a homotopy of $g_\text{res}\vert \mathbf S''$ relative to  $\partial \mathbf S''$ can be extended to a proper homotopy of $g_\text{res}$ relative to  $\mathbf S'\setminus \operatorname{int}\left(\mathbf S''\right)\ \left(\supset\partial \mathbf S'\right)$. Thus, after a homotopy of $g_\text{res}\vert \mathcal A_1'$ relative to  $\partial \mathcal A_1'$ (see \Cref{particularannuluscompression}), we may assume $g_\text{res}\left( \mathcal A_1'\right)=\mathcal C$. Applying this argument to each of $\mathcal A_2',...,\mathcal A_{m-1}'$, after a homotopy of $g_\text{res}\vert \mathcal A_1'\cup \cdots \cup \mathcal A_{m-1}'$ relative to  $\partial\left(\mathcal A_1'\cup \cdots \cup \mathcal A_{m-1}'\right)$, we may assume $g_\text{res}\left(\mathcal A_1'\cup \cdots \cup \mathcal A_{m-1}'\right)=\mathcal C$.  Further, the part $2.$ of \Cref{deg1tohomeo} tells that there exists a one-sided tubular neighborhood $\mathcal C_m'\times [1,2]\equiv\mathcal V'\subset S_{0,1,1}$ of $\mathcal C_m'\equiv\mathcal C_m'\times \{2\}$ and a one-sided tubular neighborhood $\mathcal C\times [1,2]\equiv\mathcal V$ of $\mathcal C\equiv\mathcal C\times\{2\}$ such that $g_\text{res}\left(z,r\right)=(g_\text{res}(z), r)$ for all $(z,r)\in \mathcal C_m'\times [1,2]$. Applying \Cref{pushingleft}, after a homotopy of $g_\text{res}\vert \left(\mathcal A_1'\cup \cdots \cup \mathcal A_{m-1}'\right)\cup \mathcal V'$ relative to  $\mathcal C_1'\sqcup \left(\mathcal C_m'\times \{1\}\right)$, we have $g_\text{res}^{-1}(\mathcal C)=\mathcal C_1'$ and $g_\text{res}\vert \mathcal C_1'\to \mathcal C$ is an $n$-sheeted covering map.

Since every proper homotopy of $g_\text{res}$ has been done relative to  $\partial \mathbf S'$, pasting $g\vert g^{-1}(D)\to D$ with all those proper homotopies, we can say that $g\colon \mathbb  C\to \mathbb  C$ can be properly homotoped so that $g\vert g^{-1}(\mathcal C)\to \mathcal C$ becomes an orientation-preserving $n$-sheeted covering map from a (single) circle onto a circle. Since an $n$-sheeted orientation-preserving covering $\mathbb  S^1\to \mathbb  S^1$ is of the form $\mathbb  S^1\ni z\longmapsto h(z^n)\in \mathbb  S^1$ for some orientation-preserving self-homomorphism $h\colon \mathbb  S^1\to \mathbb  S^1$, by an application of \Cref{disk}, after a proper homotopy, we may assume $g\vert g^{-1}(\mathbb  S^1)=\mathbb  S^1\ni z\longmapsto z^n\in \mathbb  S^1$. Now, applying \Cref{Alexandertrick} and \Cref{Alexandertrick2}, we can conclude that $g$, and consequently $f$ as well, is properly homotopic to $\mathbb  C\ni z\longmapsto z^n\in \mathbb  C$. This completes the proof of the case when $\deg(f)=n>0$.

Now, assume $\deg(f)=n<0$. Since the complex conjugation is an orientation-reversing self-homeomorphism of $\mathbb  C$, the map $\overline f\colon \mathbb  C\ni z\longmapsto \overline{f(z)}\in \mathbb  C$ is orientation-preserving of degree $-n$. Thus, $\overline f$ is properly homotopic to $\mathbb  C\ni z\longmapsto z^{-n}\in \mathbb  C$ by the previous case. Therefore, $f$ is properly homotopic to $\mathbb  C\ni z\longmapsto \overline z^{-n}\in \mathbb  C$.   
\end{proof}
It is well known that the annulus theorem establishes the following theorem, which is the main tool for showing that the connected sum operation of topological $n$-manifolds is independent of the choice of locally flat embeddings of the $n$-balls. The analog of \Cref{disk} for smooth manifolds is called the Cerf-Palais’ disk theorem \cite[Theorem 5.5]{MR116352} \cite[Theorem B]{MR117741}.
\begin{theorem}
    \textup{\cite[Topological Ball Embedding Theorem]{friedlbook} \cite[Theorem 1.3]{MR181989} Let $M$ be an oriented, connected $n$-dimensional topological manifold without boundary. If $\varphi, \psi\colon \mathbb  B^n\hookrightarrow M$ are two locally flat orientation-preserving embeddings of the closed unit balls into $M$, then there exists a homotopy $\mathcal H\colon M\times [0,1]\to M$ through homeomorphisms such that $\mathcal H(-,0)=\operatorname{Id}_M$ and $\mathcal H(-,1)\circ \varphi=\psi$. In particular, $\mathcal H$ is a proper homotopy.}\label{disk}
\end{theorem}

In the proof of part~\ref{plane2} of \Cref{plane}, we also use the following fact, whose proof is the same as that of \Cref{Alexandertrick2}.

 \begin{lemma}\textup{Let $\mathbb  B^2\coloneqq \{z\in \mathbb  C:0\leq|z|\leq 1\}$. Suppose $f\colon \mathbb  B^2\to \mathbb  B^2$ is a map such that $f^{-1}(\mathbb  S^1)= \mathbb  S^1$. Then $f$ can be homotoped relative to  $\mathbb  S^1$ to a map $g\colon \mathbb  B^2\to \mathbb  B^2$ such that $g(0)=0$ and $g(z)=|z|\cdot f(z/|z|)$ if $z\neq 0$.}\label{Alexandertrick}\end{lemma}

 Now we prove \Cref{puncturedplane}.

\begin{lemma}
\textup{Let $f$ be a $\pi_1$-injective proper map from $\mathbb  C^*$ to a non-compact oriented surface $\Sigma$. Suppose $\deg(f)=0$. Then $\Sigma$ contains a properly embedded essential punctured disk $\mathrm D_*$ such that $f$ can be properly homotoped so that $f(\mathbb  C^*)\subseteq \mathrm D_*$. }\label{puncturedplane1'lemma}
\end{lemma}
\begin{proof}
    Notice that $\Sigma\neq \mathbb  R^2$ because $\pi_1(f)$ is injective. So, there exists an \textup{LFCS} $\mathscr A$ on $\Sigma$ such that $\mathscr A$ decomposes $\Sigma$ into bordered subsurfaces, and a complementary component of this decomposition is homeomorphic to either $S_{1,1}$, $S_{0,3}$, or $S_{0,1,1}$. Notice that each component of $\mathscr A$ is a primitive circle on $\Sigma$. By \Cref{whiteny}, we may assume $f$ is smooth as well as transverse to $\mathscr A$. Therefore, $f^{-1}(\mathscr A)$ is an \textup{LFCS} on $\mathbb  C^*$ by \Cref{whiteny}. Since each component of $\mathscr A$ is a primitive circle on $\Sigma$, and $f$ is $\pi_1$-injective, by \Cref{deg1tohomeoremark2}, after a proper homotopy, we may further assume that each component of $f^{-1}(\mathscr A)$, if any, is a primitive circle on $\mathbb  C^*$, and $f$ maps every component of $f^{-1}(\mathscr A)$ onto a component of $\mathscr A$ by a covering map. Moreover, any two primitive circles on $\mathbb  C^*$ co-bound an annulus in $\mathbb  C^*$, but no two distinct components of $\mathscr A$ co-bound an annulus in $\Sigma$ together imply either $f^{-1}(\mathscr A)=\varnothing$ or there exists a component $\mathcal C$ of $\mathscr A$ such that $f^{-1}(\mathscr A)=f^{-1}(\mathcal C)\neq\varnothing$. 
    
     At first, suppose $f^{-1}(\mathscr A)=\varnothing$. Then $f(\mathbb  C^*)$, being non-compact, must be contained in a punctured disk $\mathrm D_*$ that appears as a complementary component of the decomposition of $\Sigma$ by $\mathscr A$. Hence, the proof for this case is complete.

    From now onwards we will assume $f^{-1}(\mathscr A)=f^{-1}(\mathcal C)\neq\varnothing$. If $f^{-1}(\mathcal C)$ has more than one component, then we want to properly homotope $f$ so that $f^{-1}(\mathcal C)$ becomes a single circle and $f\vert f^{-1}(\mathcal C)\to \mathcal C$ is a covering map. For future use, note that if $\mathrm S'$ is a compact bordered subsurface of $\mathbb  C^*$, then a homotopy of $f\vert \mathrm S'$ relative to  $\partial \mathrm S'$ can be extended to a proper homotopy of $f$.

    Notice that any two distinct components of $f^{-1}(\mathcal C)$, being primitive circles on $\mathbb  C^*$, co-bound an annulus in $\mathbb  C^*$. So, let $\mathcal A'$ be an annulus co-bounded by two components of $f^{-1}(\mathcal C)$. The continuity of $f$ tells that in the decomposition of $\Sigma$ by $\mathscr{A}$, $f(\mathcal{A'})$ lies within one of the two complementary components that share $\mathcal C$ as a common boundary. Thus, we can always find an essential compact bordered subsurface $\mathrm S$ of $\Sigma$ such that $\mathrm S$ is contained in one of the complementary components of $\Sigma$, $f(\mathcal A')\subseteq \mathrm S$, and $f(\partial \mathcal A')=\mathcal C\subseteq \partial\mathrm S$. Recall that \(f\) restricts to a covering map from each component of \(f^{-1}(\mathcal{C})\) onto \(\mathcal{C}\). Since $f\vert \mathcal A'\to \mathrm S$ is $\pi_1$-injective and $f\vert \partial \mathcal A'\to \mathcal C$ is a local homeomorphism, after a homotopy of $f\vert \mathcal A'\to \mathrm S$  relative to  $\partial \mathcal A'$, we may assume $f(\mathcal A')\subseteq \mathcal C$ \cite[Lemma 1.4.3]{MR224099}. Let $\mathcal A'_\text{out}$ denote the union of all annuli co-bounded by two components of $f^{-1}(\mathcal C)$. Then, after a homotopy of $f\vert \mathcal A'_\text{out}\to \Sigma$ relative to  $\partial \mathcal A_\text{out}'$, we may assume $f\left(\mathcal A_\text{out}'\right)\subseteq \mathcal C$. Denote the boundary components of $\mathcal A'_\text{out}$ by $\mathcal C'$ and $\mathcal C''$. Now, \Cref{deg1tohomeoremark2} tells that there exists a one-sided tubular neighborhood $\mathcal C'\times [1,2]\equiv\mathcal V'$ of $\mathcal C'\equiv\mathcal C'\times \{2\}$ and a one-sided tubular neighborhood $\mathcal C\times [1,2]\equiv\mathcal V$ of $\mathcal C\equiv\mathcal C\times\{2\}$ such that $f\left(z,r\right)=(f(z),r)$ for all $(z,r)\in \mathcal C'\times [1,2]$. Applying \Cref{pushingleft}, after a homotopy of $f\vert \mathcal A_\text{out}'\cup \mathcal V'$ relative to  $\mathcal C''\sqcup \left(\mathcal C'\times \{1\}\right)$, we can assume that $f^{-1}(\mathcal C)=\mathcal C''$ and $f\vert \mathcal C''\to \mathcal C$ is an $n$-sheeted covering map.

    At this point, two cases arise depending on whether $\Sigma$ is equal to $\mathbb  C^*$ or not. Let's begin by assuming $\Sigma\neq \mathbb  C^*$. Thus, if $\mathrm S_1$ and $\mathrm S_2$ are two distinct, complementary components of the decomposition of $\Sigma$ by $\mathscr A$ such that $\mathrm S_1$ is a punctured disk and $\partial\mathrm S_1\cap \partial \mathrm S_2\neq \varnothing$, then $\mathrm S_2$ must be compact. Since $f$ is proper, $f(\mathbb  C^*)$ is non-compact. Now, upon considering the decomposition of $\Sigma$ by $\mathscr A$, we can tell that $f\left(\mathbb  C^*\right)$ is contained in one of the two complementary components that share $\mathcal C$ as a common boundary. Hence, there exists a punctured disk $\mathrm D_*$ that appears as a complementary component of the decomposition of $\Sigma$ by $\mathscr A$ such that $\partial \mathrm D_*=\mathcal C$ and $f\left(\mathbb  C^*\right)\subseteq\mathrm D_*$.

    Next, assume $\Sigma=\mathbb  C^*$, and let $\mathrm D_*^{(1)}$ and $\mathrm D_*^{(2)}$ be the punctured disks that appear as complementary components of the decomposition of $\mathbb  C^*$ by $\mathcal C$. Note that in this case, $f(\mathbb  C^*)$ cannot be equal to $\mathbb  C^*$; otherwise, applying \Cref{Alexandertrick2} on each $f\vert f^{-1}(\mathrm D_*^{(1)})\to \mathrm D_*^{(1)}$ and $f\vert f^{-1}(\mathrm D_*^{(2)})\to \mathrm D_*^{(2)}$, would yield $\deg(f)=\pm n\neq 0$. So, $f\left(\mathbb  C^*\right)$ must contained in one of $\mathrm D_*^{(1)}$ or $\mathrm D_*^{(2)}$. So, we are done. \end{proof}
Notice that, except in the last paragraph of the previous proof, we have not utilized the fact that $\deg(f)=0$. This is because of the following proposition.
\begin{proposition}
    \textup{Let $f$ be a $\pi_1$-injective proper map from $\mathbb  C^*$ to a non-compact oriented surface $\Sigma$ such that $\deg(f)\neq 0$. Then $\Sigma=\mathbb  C^*$.}\label{infiniteindex}
\end{proposition}
\begin{proof}
    Since $\pi_1(f)$ is injective, $\Sigma\neq \mathbb  R^2$. Now, note that the fundamental group of a non-compact surface other than the plane and the punctured plane is a free group of rank at least two. By \Cref{degreeonemapsarepi1surjective}, the index $\left[\pi_1(\Sigma):\operatorname{im} \pi_1(f)\right]$ must be finite. Thus, $\pi_1(\Sigma)$ is finitely generated. So, if $\Sigma\neq \mathbb  C^*$, then by the Nielsen-Schreier index formula, $\operatorname{im}\pi_1(f)$ is a free group of rank at least $2$, a contradiction. 
\end{proof}
Now, we are ready to prove the first part of \Cref{puncturedplane}. 

\begin{proof}[Proof of part~\ref{puncturedplane1} of \Cref{puncturedplane}]
  Notice that $\Sigma\neq \mathbb  R^2$ because $f$ is $\pi_1$-injective. By applying \Cref{puncturedplane1'lemma}, we may assume that $f(\mathbb  S^1\times \mathbb  R)\subseteq \mathrm D_*$ for some properly embedded essential punctured disk $\mathrm D_*\subset \Sigma$. Let $e$ be the (isolated planar) end of $\Sigma$ determined by $\mathrm  D_*$. Consider a locally-finite, pairwise-disjoint collection $\mathscr A\coloneqq\{\mathcal C_i:i=1,2,..\}$  of smoothly embedded primitive circles on $\Sigma$ such that each $\mathcal C_i$ is contained in $\operatorname{int}(\mathrm D_*)$. Observe  that if a proper map $g$ is properly homotopic to $f$, then $g^{-1}(\mathscr A)\neq \varnothing$ because $\operatorname{Ends}(g)=\operatorname{Ends}(f)$ sends both elements of $\operatorname{Ends}(\mathbb  S^1\times \mathbb  R)$ to $e$. By \Cref{whiteny}, we may assume $f$ is smooth as well as transverse to $\mathscr A$. Since each component of $\mathscr A$ is a primitive circle on $\Sigma$, and $f$ is $\pi_1$-injective, by \Cref{deg1tohomeoremark2}, after a proper homotopy, we may further assume that each component of the non-empty \text{LFCS} $f^{-1}(\mathscr A)$ is a primitive circle on $\mathbb  S^1\times \mathbb  R$ and $f$ maps for every component of $f^{-1}(\mathscr A)$ onto a component of $\mathscr A$ by a covering map. Thus, there exists a positive integer $n_0$ such that $f^{-1}(\mathcal C_{n_0})\neq \varnothing$ and $f$ sends every component of $f^{-1}(\mathcal C_{n_0})$ onto $\mathcal C_{n_0}$ by a covering map. Now, by an argument similar to what is given in the proof of \Cref{puncturedplane1'lemma}, after a proper homotopy, we may assume that $\mathcal C_{n_0}'\coloneqq f^{-1}(\mathcal C_{n_0})$ is a single circle and $f\vert \mathcal C_{n_0}'\to \mathcal C_{n_0}$ is a finite-sheeted covering. Let $\mathcal D_*$ be the properly embedded essential punctured disk contained in $\mathrm D_*$ such that $\partial \mathcal D_*=\mathcal C_{n_0}$. Notice that $f(\mathbb  S^1\times \mathbb  R)=\mathcal D_*$. 

  Now, choose $r\neq 0$ so that $\mathbb  S^1\times \{r\}$ doesn't intersect $\mathcal C_{n_0}'$. Since $\mathbb  S^1\times \{r\}$ co-bounds an annulus with each of the circles $\mathbb  S^1\times \{0\}$ and $\mathcal C_{n_0}'$, the isotopy extension theorem \cite[Proposition 1.11]{MR2850125} \cite[Theorem 1.3.]{MR1336822} gives two homotopies $\mathcal H_1, \mathcal H_2\colon (\mathbb  S^1\times \mathbb  R)\times [0,1]\to \mathbb  S^1\times \mathbb  R$ through homeomorphisms such that $\mathcal H_1(-,0)=\operatorname{Id}_{\mathbb S^1\times \mathbb R}=\mathcal H_2(-,0)$, $\mathcal H_1(\mathcal C_{n_0}',1)=\mathbb  S^1\times \{r\}$, and $\mathcal H_2(\mathbb  S^1\times \{r\},1)=\mathbb  S^1\times \{0\}$. By \cite[Theorem 1.3]{MR181989}, $\mathcal H_1$ and $\mathcal H_2$ are proper homotopies. Thus, there exists a homeomorphism of $\mathbb  S^1\times \mathbb  R$ properly homotopic to the identity, which sends $\mathcal C_{n_0}'$ onto $\mathbb  S^1\times \{0\}$. Therefore, after a proper homotopy, we may assume that $f(\mathbb  S^1\times \mathbb  R)=\mathcal D_*$ and $f\vert \mathbb  S^1\times \{0\}\to \partial \mathcal D_*=\mathcal C_{n_0}$ is a finite-sheeted covering. Now, by the classification of finite-sheeted covering maps of the circle and together with an application of \Cref{Alexandertrick2} on each of $f\vert \mathbb  S^1\times (-\infty,0]\to \mathcal D_*$ and $f\vert \mathbb  S^1\times [0,\infty)\to \mathcal D_*$, we can conclude that there exists a $\pi_1$-injective, proper embedding $\iota\colon \mathbb  S^1\times [0,\infty)\hookrightarrow \Sigma$ with $\operatorname{im}(\iota)=\mathcal D_*$ such that after a proper homotopy, $f$ can be described by the proper map $\mathbb  S^1\times \mathbb  R\ni (z,t)\longmapsto \iota\left(z^d,|t|\right)\in \Sigma$ for some non-zero integer $d$. Moreover, by an argument similar to that given in the first part of the proof of part~\ref{plane1} of \Cref{plane}, we can conclude that given any compact subset $K$ of $\Sigma$, there exists a proper map $g$ properly homotopic to $f$ such that $\operatorname{im}(g)\subseteq \Sigma\setminus K$. 
\end{proof}

\begin{corollary}
    \textup{Let $f_1$ and $f_2$ be two \( \pi_1 \)-injective proper maps of degree zero from $\mathbb S^1\times \mathbb R$ to a non-compact oriented surface $\Sigma$. Then \( f_1 \) is properly homotopic to \( f_2 \) if and only if \( \operatorname{Ends}(f_1) = \operatorname{Ends}(f_2) \) and \( \widehat{\pi}(f_1) = \widehat{\pi}(f_2)\colon \widehat{\pi}(\mathbb{S}^1 \times \mathbb{R}) \to \widehat{\pi}(\Sigma) \), where \( \widehat{\pi} \) denotes the set of free homotopy classes of oriented loops.}
\end{corollary}

An argument similar to the one given in the proof of part~\ref{puncturedplane1} of \Cref{puncturedplane} provides the following.
\begin{proposition}
    \textup{Let $f\colon \mathbb  S^1\times \mathbb  R\to \mathbb  S^1\times \mathbb  R$ be a $\pi_1$-injective proper map of degree $0$. Then $f$ is properly homotopic to $\mathbb  S^1\times \mathbb  R\ni(z,t)\longmapsto \left(z^d,|t|\right)\in \mathbb  S^1\times \mathbb  R$ for some integer $d\neq 0$. }
\end{proposition}

We now prove the second part of \Cref{puncturedplane}, using the following lemma.

 \begin{lemma}
     \textup{Let $\varphi\colon \mathbb  C^*\to \mathbb  C^*$ be a proper map of non-zero degree. If $\varphi$ is a homotopy equivalence, then it is properly homotopic to a homeomorphism.}\label{strpuncturedpalne}
 \end{lemma}

 \begin{proof}
     Pick a smoothly embedded primitive circle $\mathcal C$ on $\mathbb  C^*$. Thus, there are punctured disks $\mathrm D_*^{(1)}$ and $\mathrm D_*^{(2)}$ such that $\mathbb  C^*=\mathrm D_*^{(1)}\cup \mathrm D_*^{(2)}$ and $\mathcal C=\mathrm D_*^{(1)}\cap \mathrm D_*^{(2)}$. By \Cref{whiteny}, after a proper homotopy, we may assume $\varphi$ is smooth as well as transverse to $\mathcal C$. Now, \Cref{non-surjectivepropermaphasdegreezero} tells that $\varphi$ remains surjective even after a proper homotopy.  Therefore, $\varphi^{-1}(\mathcal C)$ is a non-empty, pairwise-disjoint, finite collection of smoothly embedded circles on $\mathbb  C^*$. Further, after a proper homotopy, similar to what we have done in the proof of part~\ref{plane2} of \Cref{plane}, we may assume $\mathcal C'\coloneqq\varphi^{-1}(\mathcal C)$ is a primitive circle on $\mathbb  C^*$ and $\varphi\vert \mathcal C'\to \mathcal C$ is a homeomorphism. This is possible as $\varphi$ is a homotopy equivalence.  Let $\mathrm D_*^{(1)'}$ and $\mathrm D_*^{(2)'}$ be the punctured disks such that $\mathbb  C^*=\mathrm D_*^{(1)'}\cup \mathrm D_*^{(2)'}$ and $\mathcal C'=\mathrm D_*^{(1)'}\cap \mathrm D_*^{(2)'}$. Since $\varphi$ is surjective, possibly after re-indexing, we may assume  $\varphi^{-1}(\mathrm D_*^{(j)})=\mathrm D_*^{(j)'}$ for $j=1,2$. Finally, to conclude apply \Cref{Alexandertrick2} on $\varphi\vert\mathrm D_*^{(j)'}\to \mathrm D_*^{(j)}$ for $j=1,2$.
 \end{proof}
 \begin{remark}
     \textup{The non-zero degree assumption in \Cref{strpuncturedpalne} can't be dropped. For example, $f\colon \mathbb  S^1\times \mathbb  R\ni (z,t)\longmapsto \left(z, |t|\right)\in \mathbb  S^1\times \mathbb  R$ is a self-homotopy equivalence of $\mathbb  C^*$ which is not properly homotopic to any homeomorphism of $\mathbb  C^*$ as $\deg(f)=0$.}
 \end{remark}
 
 So we now prove part~\ref{puncturedplane2} of \Cref{puncturedplane}.

\begin{proof}[Proof of part~\ref{puncturedplane2} of \Cref{puncturedplane}]
 By \Cref{infiniteindex}, $\Sigma=\mathbb  C^*$. Let $p\colon \widetilde \Sigma\to \mathbb  C^*$ be the covering corresponding the subgroup $\operatorname{im}\pi_1(f)$ of $\pi_1(\mathbb  C^*)$, and let $\widetilde f\colon \mathbb  C^*\to \widetilde \Sigma$ be a lift of $f$ w.r.t. $p$, i.e., $p\circ \widetilde f=f$. Thus, $\operatorname{im}\pi_1(p)=\operatorname{im}\pi_1(f)$, and hence $\pi_1(\widetilde f)$ is an isomorphism because a covering map induces injection between fundamental groups. In particular, the fundamental group of $\widetilde \Sigma$ is infinite cyclic. Hence, $\widetilde \Sigma=\mathbb  C^*$. The properness of $f$ implies the properness of $\widetilde f$ by \Cref{techlemma}. Since non-compact surfaces are $K(\pi,1)$ CW-complexes, by Whitehead theorem $\widetilde f$ is a homotopy equivalence. By \Cref{techlemma}, we may assume $p$ is a $d$-sheeted covering for some positive integer $d$ and $\widetilde \Sigma$ is orientable.  Fix an orientation of $\widetilde \Sigma$. By \Cref{degreeonemapchecking}, $\deg(p)=\pm d$. Now, $\deg(\widetilde f)\neq 0$ because $0\neq n=\deg(f)=\deg(p\widetilde f)=(\pm d)\cdot \deg(\widetilde f)$. Thus, $\widetilde f\colon \mathbb  C^*\to \mathbb  C^*$ is a homotopy equivalence and $\deg(\widetilde f)\neq 0$. So $\widetilde f$ is properly homotopic to a homeomorphism by \Cref{strpuncturedpalne}, i.e., $\deg(\widetilde f)=\pm 1$. Hence, $n=\pm d$, and $f$ is properly homotopic to a $d$-sheeted covering.

 At first, suppose $n>0$. By the previous paragraph, without loss of generality, we may assume $f$ is an $n$-sheeted covering map. Now, covering space theory \cite[Proposition 1.37]{MR1867354} gives a self-homeomorphism $h$ of $\mathbb  C^*$ such that $f(z)=h(z^n)$ for all $z\in \mathbb  C^*$. Certainly, $h$ is orientation-preserving. By \Cref{homogeneity} below, we may assume $f(1)=1=h(1)$. Now, \cite[Theorem 5.7]{MR214087} tells that there exists  a level-preserving homeomorphism $\mathcal H\colon \mathbb  C^*\times [0,1]\to\mathbb  C^*\times [0,1]$ which agrees with $h$ on $\mathbb  C^*\times 0$ and with $\operatorname{Id}_{\mathbb  C^*}$ on $\mathbb  C^*\times1$. The projection $\mathbb  C^*\times [0,1]\to \mathbb  C^*$ is proper implies $h$ is properly homotopic to $\operatorname{Id}_{\mathbb  C^*}$. So we are done when $\deg(f)=n>0$. 

    Now, assume $\deg(f)=n<0$. Since the complex conjugation is an orientation-reversing self-homeomorphism of $\mathbb  C$, the map $\overline f\colon \mathbb  C^*\ni z\longmapsto \overline{f(z)}\in \mathbb  C^*$ is orientation-preserving of degree $-n$. Thus, $\overline f$ is properly homotopic to $\mathbb  C^*\ni z\longmapsto z^{-n}\in \mathbb  C^*$ by the previous case. Therefore, $f$ is properly homotopic to $\mathbb  C^*\ni z\longmapsto \overline z^{-n}\in \mathbb  C^*$.
 \end{proof}

It is known that every connected, boundaryless manifold $M$ is homogeneous, i.e., for any two points $x$ and $y$ of $M$, there exists a homeomorphism $h$ of $M$ sending $x$ to $y$. Moreover, \Cref{disk} tells us that the homeomorphism $h$ can be chosen so that it becomes properly homotopic to the identity of $M$.

\begin{proposition}
    \textup{Let $M$ be an orientable, connected topological manifold without boundary. If $x$ and $y$ are two points of $M$, then there exists a self-homeomorphism $\varphi\colon M\to M$ properly homotopic to the identity of $M$ such that $\varphi(x)=y$.}\label{homogeneity}
\end{proposition}
\section*{Acknowledgments}
This article is largely based on the author’s Ph.D. thesis completed at the Indian Institute of Science, Bangalore, under the supervision of Prof. Siddhartha Gadgil, whose guidance was essential throughout the work. The author gratefully acknowledges the anonymous referees for their thorough reading of the manuscript and for their constructive comments and suggestions, including the recommendation to incorporate \Cref{ref}. The author also thanks Prof. Claude L. Schochet for his helpful correspondence concerning the literature on the cardinality of the first derived limit.

\bibliographystyle{plain}
\bibliography{bibliography.bib}

\begin{thebibliography}{10}

\bibitem{MR1308972}
Mark Brittenham.
\newblock {\href{https://doi.org/10.1016/0166-8641(94)90008-6}{\color{black}Essential laminations and deformations of homotopy equivalences: from essential pullback to homeomorphism}}.
\newblock {\em Topology Appl.}, 60(3):249--265, 1994.

\bibitem{neverpublished}
Mark Brittenham.
\newblock {{$\pi_1$}-injcetive, proper maps of open surfaces}, unpublished (1989).

\bibitem{MR334225}
E.~M. Brown and T.~W. Tucker.
\newblock {\href{https://doi.org/10.2307/1996769}{\color{black}On proper homotopy theory for noncompact {$3$}-manifolds}}.
\newblock {\em Trans. Amer. Math. Soc.}, 188:105--126, 1974.

\bibitem{MR181989}
R.~H. Crowell.
\newblock {\href{https://doi.org/10.2307/2034296}{\color{black} Invertible isotopies}}.
\newblock {\em Proc. Amer. Math. Soc.}, 14:658--664, 1963.

\bibitem{MR4843735}
Sumanta Das.
\newblock {\href{https://doi.org/10.2140/agt.2024.24.4423}{\color{black}Strong topological rigidity of noncompact orientable surfaces}}.
\newblock {\em Algebr. Geom. Topol.}, 24(8):4423--4469, 2024.

\bibitem{MR214087}
D.~B.~A. Epstein.
\newblock {\href{https://doi.org/10.1007/BF02392203}{\color{black}Curves on {$2$}-manifolds and isotopies}}.
\newblock {\em Acta Math.}, 115:83--107, 1966.

\bibitem{MR192475}
D.~B.~A. Epstein.
\newblock {\href{https://doi.org/10.1112/plms/s3-16.1.369}{\color{black}The degree of a map}}.
\newblock {\em Proc. London Math. Soc. (3)}, 16:369--383, 1966.

\bibitem{MR2850125}
Benson Farb and Dan Margalit.
\newblock {\em {\href{https://mathscinet.ams.org/mathscinet-getitem?mr=2850125}{\color{black}A primer on mapping class groups}}}, volume~49 of {\em Princeton Mathematical Series}.
\newblock Princeton University Press, Princeton, NJ, 2012.

\bibitem{friedlbook}
Stefan Friedl.
\newblock {\em Topology (Course notes)}.
\newblock AMS Open Math Notes, July 2023.
\newblock Reference OMN:202307.111368. Available at \href{https://www.ams.org/open-math-notes/omn-view-listing?listingId=111368}{https://www.ams.org/open-math-notes/omn-view-listing?listingId=111368}.

\bibitem{MR2365352}
Ross Geoghegan.
\newblock {\em {\href{https://doi.org/10.1007/978-0-387-74614-2}{\color{black}Topological methods in group theory}}}, volume 243 of {\em Graduate Texts in Mathematics}.
\newblock Springer, New York, 2008.

\bibitem{MR275436}
Martin~E. Goldman.
\newblock {\href{https://doi.org/10.2307/1995610}{\color{black}An algebraic classification of noncompact {$2$}-manifolds}}.
\newblock {\em Trans. Amer. Math. Soc.}, 156:241--258, 1971.

\bibitem{MR3598162}
Craig~R. Guilbault.
\newblock {\href{https://doi.org/10.1007/978-3-319-43674-6_3}{\color{black}Ends, shapes, and boundaries in manifold topology and geometric group theory}}.
\newblock In {\em Topology and geometric group theory}, volume 184 of {\em Springer Proc. Math. Stat.}, pages 45--125. Springer, [Cham], 2016.

\bibitem{MR1867354}
Allen Hatcher.
\newblock {\em \href{https://mathscinet.ams.org/mathscinet-getitem?mr=1867354}{\color{black}Algebraic topology}}.
\newblock Cambridge University Press, Cambridge, 2002.
\newblock Available at \href{https://pi.math.cornell.edu/~hatcher/AT/AT.pdf}{https://pi.math.cornell.edu/~hatcher/AT/AT.pdf}.

\bibitem{MR1336822}
Morris~W. Hirsch.
\newblock {\em {\href{https://doi.org/10.1007/978-1-4684-9449-5}{\color{black}Differential topology}}}, volume~33 of {\em Graduate Texts in Mathematics}.
\newblock Springer-Verlag, New York, 1994.
\newblock Corrected reprint of the 1976 original.

\bibitem{MR2553079}
S.~A. Melikhov.
\newblock {\href{https://doi.org/10.1070/RM2009v064n03ABEH004620}{\color{black}Steenrod homotopy}}.
\newblock {\em Uspekhi Mat. Nauk}, 64(3(387)):73--166, 2009.

\bibitem{MR116352}
Richard~S. Palais.
\newblock {\href{https://doi.org/10.2307/1993171}{\color{black}Natural operations on differential forms}}.
\newblock {\em Trans. Amer. Math. Soc.}, 92:125--141, 1959.

\bibitem{MR117741}
Richard~S. Palais.
\newblock {\href{https://doi.org/10.2307/2032968}{\color{black}Extending diffeomorphisms}}.
\newblock {\em Proc. Amer. Math. Soc.}, 11:274--277, 1960.

\bibitem{MR254818}
Richard~S. Palais.
\newblock {\href{https://doi.org/10.2307/2037337}{\color{black}When proper maps are closed}}.
\newblock {\em Proc. Amer. Math. Soc.}, 24:835--836, 1970.

\bibitem{MR224099}
Friedhelm Waldhausen.
\newblock {\href{https://doi.org/10.2307/1970594}{\color{black}On irreducible {$3$}-manifolds which are sufficiently large}}.
\newblock {\em Ann. of Math. (2)}, 87:56--88, 1968.

\end{thebibliography}

\end{document}